\documentclass[11pt]{article}%
\usepackage{hyperref}
\usepackage[ansinew]{inputenc}
\usepackage{amssymb}
\usepackage{amsmath}
\usepackage[english]{babel}
\usepackage{cases}
\usepackage{setspace}
\newtheorem{theorem}{Theorem}

\newtheorem{corollary}[theorem]{Corollary}

\newtheorem{definition}[theorem]{Definition}

\newtheorem{lemma}[theorem]{Lemma}

\newtheorem{proposition}[theorem]{Proposition}
\newtheorem{remark}[theorem]{Remark}

\newenvironment{proof}[1][Proof]{\textbf{#1.} }{\ \rule{0.5em}{0.5em}}
\hypersetup{
backref=true, 
pagebackref=true,
hyperindex=true, 
colorlinks=true, 
breaklinks=true, 
urlcolor= blue, 
linktoc=section
linkcolor= blue, 
bookmarks=true, 
}
\begin{document}

\title{On the pathwise uniqueness of solutions of one-dimensional stochastic differential equations with jumps}
\maketitle{}
\author{\begin{center} {M. Benabdallah, S. Bouhadou, Y. Ouknine }\footnote{Corresponding authors \\Email address:{\bf bmohsine@gmail.com}\\
Department of mathematics, Ibn Tofail University, Kenitra, B.P. 133, Morocco\\
 Email addresses: {\bf ouknine@ucam.ac.ma} (Y. Ouknine), {\bf sihambouhadou@gmail.com} \\ LIBMA Laboratory, Department of Mathematics, Faculty of Sciences Semlalia, Cadi Ayyad University, P.B.O. 2390 Marrakesh, Morocco.\\}$^{,}$ \footnote{This work is supported by Hassan II Academy of Sciences and Technology.\\
Email address:}
\end{center}}

\date{}
\begin{spacing}{1.1}
\begin{abstract}
We consider one-dimensional stochastic differential equations with jumps in the
general case. We introduce new technics based on local time and we prove new
results on pathwise uniqueness and comparison theorems. Our approach are very
easy to handled and don’t need any approximation approach. Similar equations
without jumps were studied in the same context by \cite{Le Gall}, \cite{Ouknine} and others authors.
As an application we get a new condition on the pathwise uniqueness for the
solutions to stochastic differential equations driven by a symmetric stable Lévy
processes.
\end{abstract}

\section{Introduction}
Stochastic differential equations play a central role in the theory of stochastic processes and
are often used in the modeling of various random processes in nature. Often defined as strong solutions of stochastic differential equations, diffusion processes are
widely used in stochastic modeling e.g., Ornstein and Uhlenbeck \cite{Ornstein}
used their process for the analysis of velocity of a particle in a
fluid under the bombardment by
molecules. Samuelson \cite{samwe} introduced geometric Brownian motion for modeling the behavior of
financial markets. Also diffusions processes
appear e.g., in
stochastic population modeling. \\
In the recent years, jump processes were also used in many fields as flexible models to
describe various phenomena. In particular, they are frequently used in financial
modeling, it would be nice for the readers to refer to R. Cont
and P. Tankov \cite{Rama} and their reference for details about applications.
Barndorff-Nielsen \cite{bandroff} proposed the
idea for generalizing diffusion processes by means of changing the driving Wiener process by a Lévy process and defined the so-called background driven Ornstein-Uhlenbeck type process.\\
In the present paper, we consider stochastic differential equations driven by both a Wiener
process and a Poisson random measure, and study the question of pathwise uniqueness of this class
called stochastic differential equation with jumps:
\begin{equation}
 dX_t=\sigma(X_t)dW_t+ b(X_t)dt + \int F(X_{t-},z)(\mu-\nu)(dz,dt),\ \ 	X_0=x_0\label{1}
\end{equation}
In fact, the results on pathwise uniqueness of (\ref{1})
have been obtained under Lipschitz conditions, see Skorohod \cite{skoro}, Ikeda and Watanbe \cite{Ikeda}, Protter \cite{Protter}.
In absence of jumps, this SDE
$$dX_t=\sigma(X_t)dW_t+ b(X_t)dt $$
with non-Lipschitz coefficient were considered by several authors.
There were many works which discuss under which conditions  on $b$ and $\sigma$, we have the existence of strong solutions of stochastic differential equations. In the case when the equation is one-dimensional and $\sigma$ is not degenerated, several results have been obtained by Y. Ouknine \cite{Nakao}, \cite{fsemimarting}, \cite{oukref}. For SDE's which involve local times of unknown process, the most general result is given by  M. Rutkowski \cite{Rutkowski} where he showed the so called (LT) condition is sufficient to have pathwise uniqueness. So, the purpose of this paper is to give the analogue of this condition for the one dimensional SDE with jumps which concerns the couple of coefficient $\sigma$ and $F$.
\\

This paper is arranged as follows. In section $2$, we recall the definition of (LT) condition introduced by Barlow and Perkins \cite{Barlow}, and we introduce the new definition of local time condition $(\mathcal{L}T)$.In section $3$, we give some sufficient assumptions which ensure this condition. On the other hand, we investigate this definition to prove our mean result: Pathwise uniqueness of SDE (\ref{1}). Section $4$, is devoted to generalize Bass's result \cite{Bass 2}. In the same section we derive a generalization of comparison theorem proved by M. Benabdallah, S. Bouhadou and Y. Ouknine \cite{Benabd} in the case of SDE's without jumps. In the last section, we study
the pathwise uniqueness for stochastic differential equations
driven by spectrally positive Lévy noises which arise naturally
in the study of branching processes.

\section{Preliminaries}
On some stochastic basis $(\Omega,\mathcal{A},\Bbb{F}=(\mathcal{F}_t)_{t\geq 0},P)$, we consider one-dimensionnal $\Bbb{F}$-Brownian motion $W=(W_t)_{t\geq 0}$ and a $\Bbb{F}$-Poisson point process $\mu(ds,dy)$ in $(0,\infty)\times \Bbb {R}$. We suppose that $\mu$ and $W$ are independent. The random measure $\mu(ds,dy)$ has deterministic intensity
	\[\nu (ds,dy)=ds \lambda(dy)\ \ \ \ \mbox{on}\ \ \ (0,\infty)\times \Bbb {R}
\]
where $\lambda$ is Lebesgue measure on $\Bbb{R}$.\\

A solution of \eqref{1} is any process $X=(X_t)_{t\geq 0}$ on $(\Omega,\mathcal{A},\Bbb{F},P)$ satisfiying i) and ii) below:

\begin{enumerate}
	\item [i)]$X$ is adapted $\Bbb{F}$-adapted and càdlàg
	\item [ii)]
	\[X_t=X_0+\int_0^t \sigma(X_s)dW_s + \int_0^t b(X_s)ds +\int_0^t \int F(X_{s-},z)(\mu -\nu)(dz,ds)
\]
\end{enumerate}

We say that the pathwise uniqueness of solutions for \eqref{1} holds if whenever $X$ and $X'$ are any two solutions defined on the same stochastic basis $(\Omega,\mathcal{A},\Bbb{F},P)$ with the same $\Bbb{F}$-Brownian motion $W=(W_t)_{t\geq 0}$ and the same  $\Bbb{F}$-Poisson point process $\mu(ds,dy)$ such that $X_0=X'_0$ a.s., then $X_t=X'_t$ for all $t\geq 0$ a.s.\\

\mbox{}

We present (as in Protter  \cite {Protter}) the notion of local time of semimartingale. If $X$ is a general càdlàg semimartingale, let $\Delta X$ denote the process $\Delta X_t=X_t-X_{t-}$.\\
We recall the quadratic variation pocesses of $X$ is defined by
	\[\left[ X \right]_t=X_t^2-2\int_0^t X_{s-}dX_s
\]
The local time at $a$ of $X$, denoted $L_t^a=L^a(X)_t$ is defined to be the process given by
	
\begin{eqnarray}
	L_t^a&=&\left|X_t-a\right|-\left|X_0-a\right|-\int_0^t \mbox{sign}(X_{s-}-a)dX_s \ \ \ \ \ \ \ \ \ \ \ \ \ \ \ \ \ \ \ \ \ \ \ \ \ \ \ \ \ \ \ \ \ \ \nonumber \\
                    & & - \sum_{0<s\leq t}\left\{\left|X_s-a\right|-\left|X_{s-}-a\right|- \mbox{sign}(X_{s-}-a)\Delta X_s\right\}  \nonumber         	
\end{eqnarray}
where
$$
 \mbox{sign} (x)=\left\{\begin{array}{rl}
1 & \mbox{if}\ x>0 \\
-1 & \mbox{if}\ x\leq 0 \end{array}\right.
$$
The local time gives a generalization of Itô's formula: if $f$ is the difference of two convex functions and $f'$ is its left derivative and let $\mu$ be the signed measure which is the second derivative of $f$. Thus we have
\begin{eqnarray}
	f(X_t)&=&f(X_0) + \int_{0^+}^t f'(X_{s-})dX_s + \sum_{0<s\leq t}\left\{f(X_s)-f(X_{s-})- f'(X_{s-})\Delta X_s\right\} \nonumber \\
                    & & +\frac{1}{2}\int_{-\infty}^{\infty} L_t^a \mu(da) \nonumber         	
\end{eqnarray}

We introduce the local time slanted $\mathcal{L}_t^a(X)$ of a semimartingale $X$ by
	\[\mathcal{L}_t^a = \left|X_t-a\right|-\left|X_0-a\right|-\int_0^t \mbox{sign}(X_{s-}-a)dX_s
\]
We remark that if $f$ is the difference of two convex functions, we have
	\[f(X_t)= f(X_0) + \int_{0^+}^t f'(X_{s-})dX_s + \frac{1}{2}\int_{-\infty}^{\infty} \mathcal{L}_t^a \mu(da)
\]

Finally, we indicate the formula of occupation density, if $f$ is a bounded Borel measurable function, then, \textit{a.s.}
	\[\int_{-\infty}^{\infty} L_t^a f(a) da=\int_0^t f(X_{s-})d\left[ X \right]^c_s.
\]
where $\left[ X \right]^c$ is the continuous part of $\left[ X \right]$.\\
We apply this notation:\\
For all $x$ and $y$ in $\Bbb{R}$, $x\wedge y= \inf(x,y)$  and $x \vee y=\sup(x,y)$.\\
Now, we introduce two definitions of the (LT) condition, the first concerns the coefficient $\sigma$ uniquely and the second concerns the couple of coefficients $(\sigma, F)$ that will help us get the pathwise uniqueness of equation \eqref{1}.\\

\begin{definition}
We say that a  coefficient $\sigma $ of equation \eqref{1} satisfies (LT) condition if for two solutions $X^1$ and $X^2$ of (1) , then

\begin{equation}
 \forall t\geq 0\ \ L^0_t(X^1-X^2)=0 \label{2}
\end{equation}

\end{definition}

Now, we define the ($\mathcal{LT}$) condition concerning the couple of coefficient $(\sigma, F)$ of equation \eqref{1}.\\
\begin{definition}
We say that the  coefficients $(\sigma, F) $ of equation \eqref{1} satisfy ($\mathcal{LT}$) condition if for two solutions $X^1$ and $X^2$ of (1) , then

\begin{equation}
 \forall t\geq 0\ \ \mathcal{L}^0_t(X^1-X^2)=0 \label{3}
\end{equation}
\end{definition}

\mbox{}

We can remark that if the coefficients of \eqref{1} verify ($\mathcal{LT}$) condition then they verify the (LT) condition too which was used by several authors ( Le Gall \cite {Le Gall}, Ouknine \cite {Ouknine},...) and which permits to prove the pathwise uniqueness of solutions of one-dimensionnal stochastic differential equations without jumps.\\

We can apply our results to stochastic differential equation driven by symmetric stable processes studied by \cite {Bass 2}, taken by \cite {Bel}.\\

In the following, we prove that if  pathwise uniqueness holds for \eqref{1} before the first big jump(for example $\left|\Delta X\right|\geq 1$), then  pathwise uniqueness holds for every $t\geq 0$. This allows to consider equation \eqref{1} with only the small jumps.\\

\begin{lemma}

If we have pathwise uniqueness for equation \eqref{1} with only the small jumps, then pathwise uniqueness holds for equation \eqref{1} for general case.
\end{lemma}

\begin{proof}

Let $X$ and $Y$ two solutions of equation \eqref{1} with the same initial value. Let $S_1$ be the first time when the big jump happens($\left|\Delta X\right|\geq 1$), then we have
	\[X_t=Y_t\ \ \ \ \mbox{a.s.}\ \mbox{for}\ t\in \left[0,S_1\right)
\]
We have also $X_{S_1}=X_{S_1^-}+F(X_{S_1^-},\Delta X_{S_1})=Y_{S_1^-}+F(Y_{S_1^-},\Delta Y_{S_1})=Y_{S_1}$

We consider the filtration $\Bbb{F}^{S_1}=(\mathcal{F}_{S_1+t})_{t\geq 0}$, the $\Bbb{F}^{S_1}$-Brownian motion $W^{S_1}=(W_{S_1+t}-W_{S_1})_{t\geq 0}$ and a $\Bbb{F}^{S_1}$-Poisson point process $\mu^{S_1}(ds,dy)$  with intensity $\nu^{S_1}=ds\lambda(dy)$ on in $(0,\infty)\times \Bbb {R}$ defined by $\mu^{S_1}([0,t]\times \cdot)=\mu([S_1,S_1+t]\times \cdot)$. We consider the new equation:
\begin{equation}
 dX^{S_1}_t=\sigma(X^{S_1}_t)dW^{S_1}_t+ b(X^{S_1}_t)dt + \int F(X^{S_1}_{t-},z)(\mu^{S_1}-\nu^{S_1}),\ \ t\geq 0 	 \label{4}
\end{equation}
We consider the processes $X^{S_1}$ and $Y^{S_1}$ defined by $X^{S_1}_t=X_{S_1+t},t\geq 0$ and $Y^{S_1}_t=Y_{S_1+t},t\geq 0$
Then $X^{S_1}$ and $Y^{S_1}$ are solutions of \eqref{4} with the same condition initial and by hypothesis of pathwise uniqueness property of stochastis differential equation without the big jumps,they are equal until the first big jump when happens at time $S_2$. We have also
\[X^{S_1}_t=Y^{S_1}_t\ \ \ \ \mbox{a.s.}\ \mbox{for}\ t\in \left[0,S_2-S_1\right)
\]
This implies
\[X_t=Y_t\ \ \ \ \mbox{a.s.}\ \mbox{for}\ t\in \left[0,S_2\right)
\]
We repeat the same process, we obtain the pathwise uniqueness for equation \eqref{1} for the general case.
\end{proof}

Since in any finite time, the solution of \eqref{1} can have only a finite number of jumps of size greather than $1$. Then the large jumps do not affect the existence uniqueness of solutions and we work  uniquely with the small jumps.


\section{The pathwise uniqueness property for equation \eqref{1}}
\subsection{Main result}
Throughout this paragraph, we make the following assumption on the coefficients in equation \eqref{1}:\\
(A) the functions $\sigma$ and $b$ are measurables and bounded,\\
(B) $\left|b(x)- b(y)\right|\leq c \left|x-y\right|$ for all $x, y$,\\
(C) $\left|b(x)\right|^2 + \left|\sigma(x)\right|^2 + \int \left|F(x,z)\right|^2 \lambda(dz)\leq c (1 + \left|x\right|^2) $ for all $x$.

\mbox{}

Before giving the main result, we present a sufficent assumption to get the ($\mathcal{LT}$) condition (cf.Definition 2):

\begin{proposition}

Suppose that there exist a sequence of non-negative and twice continuously differentiable functions $\left\{\phi_n\right\}$ with the following properties:
\begin{enumerate}
	\item[(a)] $\phi_n (z) \uparrow \left|z\right|$ as $n\rightarrow \infty$;
	\item[(b)] $\left|\phi'_n(z)\right|\leq 1$ for all $z$;
	\item[(c)] $\phi''_n(z)\geq 0$ for $z \in \Bbb{R}$ and as $n\rightarrow \infty$,
	\[\phi''_n(x-y)\left[\sigma(x)-\sigma(y)\right]^2\rightarrow 0 \;\mbox{uniformly in  } \left|x\right|,\left|y\right|\leq m;
\]
  \item[(d)] as $n\rightarrow \infty$
	\[ \int \left[\phi_n (x+F(x,z)-y-F(y,z))-\phi_n (x-y)-\phi'_n (x-y)(F(x,z)-F(y,z))\right]\lambda(dz)\rightarrow 0
\]
$\mbox{uniformly in  } \left|x\right|,\left|y\right|\leq m;$
\end{enumerate}
then $(\sigma,F)$ verify ($\mathcal{LT}$) condition.
\end{proposition}

\begin{proof}[Proof of the proposition 4 ]

Let two solutions $X^1$ and $X^2$ of \eqref{1} defined on the same stochastic basis $(\Omega,\mathcal{A},\Bbb{F},P)$ with the same $\Bbb{F}$-Brownian motion $W=(W_t)_{t\geq 0}$ and the same  $\Bbb{F}$-Poisson point process $\mu(ds,dy)$ such that $X^1_0=X^2_0$ a.s., we set:

\mbox{}

$
\begin{array}{l}
I_n(s)=\int \left\{\phi_n((X^1_{s-}-X^2_{s-})+\left[F(X^1_{s-},z)-F(X^2_{s-},z)\right])\right.\\
	 \\
	 -\left.\phi_n(X^1_{s-}-X^2_{s-})-\left[F(X^1_{s-},z)-F(X^2_{s-},z)\right]\phi'_n(X^1_{s-}-X^2_{s-})\right\}\lambda(dz) 	 
\end{array}
$

\mbox{}

\mbox{}

Let $\tau_m =\inf\left\{t\geq 0: \left|X_1(t)\right|\geq m \:\mbox{or} \:\left|X_2(t)\right|\geq m \right\}$.
By application of Itô formula, we have

\mbox{}

$
\begin{array}{lll}
	 \int^{t\wedge \tau_m}_{0}I_n(s)ds &=&\phi_n(X^1_{t\wedge \tau_m}-X^2_{t\wedge \tau_m})-\phi_n(0) \\
	 & & \\
	 & & -\int_0^{t\wedge \tau_m}\phi_n'(X^1_{s-}-X^2_{s-})(b(X^1_{s-})-b(X^2_{s-}))ds\\
	 & & \\
	 & & -\int_0^{t\wedge \tau_m}\phi_n'(X^1_{s-}-X^2_{s-})(\sigma(X^1_{s-})-\sigma(X^2_{s-}))dW_s\\
	 & & \\
	 & & -\frac{1}{2}\int_0^{t\wedge \tau_m}\phi_n''(X^1_{s-}-X^2_{s-})(\sigma(X^1_{s-})-\sigma(X^2_{s-}))^2ds\\
	 & & \\
	 & & -\int_0^{t\wedge \tau_m}\int \left[\phi_n(X^1_{s-}-X^2_{s-}+(F(X^1_{s-},z)-F(X^2_{s-},z)))\right.\\
	 & & \\
	 & &  \ \ \left.-\phi_n(X^1_{s-}-X^2_{s-}) \right ](\mu -\nu)(dz,ds) 	
\end{array}
$

\mbox{}
\mbox{}
\mbox{}

According to the assumptions $(a),(c),(d)$, the  term on the left hand side, the second and the fifth terms on the right hand side tend  to zero. We obtain:

\mbox{}

$
\begin{array}{lll}
	 0 &=&\left|(X^1_{t\wedge \tau_m}-X^2_{t\wedge \tau_m})\right| \\
	 & & \\
	 & & -\int_0^{t\wedge \tau_m}\mbox{sign}(X^1_{s-}-X^2_{s-})(b(X^1_{s-})-b(X^2_{s-}))ds\\
	 & & \\
	 & & -\int_0^{t\wedge \tau_m}\mbox{sign}(X^1_{s-}-X^2_{s-})(\sigma(X^1_{s-})-\sigma(X^2_{s-}))dW_s\\
	 & & \\
	 & & -\int_0^{t\wedge \tau_m}\int \left[\left|X^1_{s-}-X^2_{s-}+(F(X^1_{s-},z)-F(X^2_{s-},z))\right|\right.\\
	 & & \\
	 & &  \ \ \left.-\left|X^1_{s-}-X^2_{s-} \right|\right ](\mu -\nu)(dz,ds) 	
\end{array}
$

\mbox{}

 We note that this leads to the formula of Tanaka which lacks a term that can only be zero in this case. We obtain

	\[\left|X^1_{t\wedge \tau_m}-X^2_{t\wedge \tau_m}\right|-\int_0^{t\wedge \tau_m} \mbox{sign}(X^1_{s-}-X^2_{s-})d(X^1_s-X^2_s)=0
\]
	
Since $ \tau_m \rightarrow \infty$ as $m\rightarrow \infty$, we obtain
	
	\[\mathcal{L}_t^0(X^1-X^2) = \left|X^1_t-X^2_t\right|-\int_0^t \mbox{sign}(X^1_{s-}-X^2_{s-})d(X^1_s-X^2_s)=0
\]
Thus $(\sigma,F)$ verify ($\mathcal{LT}$) condition.
\end{proof}

\begin{theorem}
 If $(\sigma,F)$ verify ($\mathcal{LT}$) condition,
  then the solution to \eqref{1} is pathwise unique.
\end{theorem}

Before proving theorem 5, we give several conditions which ensures $\mathcal{LT}$ condition in particular the coefficients can be discontinuous.

\begin{corollary}
 If $\sigma$ verify the following condition:
	\[\left|\sigma(x)-\sigma(y)\right|\leq h(\left|x-y\right|) \mbox{ for all } x,y
\]
  where $h:\left[ 0,\infty \right)\rightarrow \left[ 0,\infty \right)$ is continuous and nondecreasing, $h(0)=0,\\h(x)>0 $ for $x>0$, and
\[
  \int^{\epsilon}_{0}\frac{du}{h^2(u)}=\infty \mbox{ for every }\epsilon >0,
\]
and if
	$ x \rightarrow x + F(x,z)$ is nondecreasing in a neighborhood of $0$, $\lambda(dz)$ a.e.
then $(\sigma,F)$ verify ($\mathcal{LT}$) condition.
\end{corollary}

This corollary  generalizes the results of Fu-Li \cite{Fu} namely  theorem 3.2 and  3.3 because our assumption on $F$ is weaker than \cite{Fu}.

We present a second corollary in the spirit of Nakao (LeGall\cite{Le Gall}) result but for discontinuous SDE.

\begin{corollary}
If $\sigma$ verifies the following condition:
\[
 \left|\sigma(x)-\sigma(y)\right|^2\leq \left|f(x)-f(y)\right| \mbox{ for all } x,y
\]
where $f: \Bbb{R} \rightarrow  \Bbb{R}$ is nondecreasing, and there exist $\epsilon > 0$ such that
\[
  \sigma (x) > \epsilon  \mbox{ for every } x \in \Bbb{R},
\]
and if
	$ x \rightarrow x + F(x,z)$ is nondecreasing in a neighborhood of $0$, $\lambda(dz)$ a.e.
then $(\sigma,F)$ verify ($\mathcal{LT}$) condition.
\end{corollary}

\begin{proposition}
If $\sigma$  satisfies (LT) condition and $F$ verifies the following condition:
\begin{equation}
 \int \left|F(x,z)- F(y,z)\right|\lambda(dz)\leq c \left|x-y\right| \mbox{ for all } x,y \label{5}
\end{equation}	
then the solution to \eqref{1} is pathwise unique.
\end{proposition}


\subsection{Proofs of results}

In the subsection, we prove the results of the subsection below.

\begin{proof}[Proof of the  Theorem 5]

The main idea is to use the Tanaka formula. Let $X^1$ and $X^2$ be two solutions to \eqref{1},  then
\begin{equation}
\begin{array}{lll}
	\left|X^1_t-X^2_t\right|&=&\int^{t}_{0}\mbox{sign}(X^1_{s-}-X^2_{s-})d(X^1_{s}-X^2_{s})\\
		\label{6}       	
\end{array}
\end{equation}
Now, we have
\begin{equation}
\begin{array}{lll}
	d(X^1_{t}-X^2_{t})&=&(\sigma (X^1_t)-\sigma(X_t^2))dW_t + (b(X^1_t)-b(X_t^2))dt \\
	& &\\
	& &+ \int (F(X^1_{t-},z)-F(X^2_{t-},z))(\mu - \nu )(dt,dz)   \label{7}    	
\end{array}
\end{equation}

If we substitute \eqref{7} in \eqref{6}, we obtain

\mbox{}

$\begin{array}{lll}
\left|X^1_t-X^2_t\right|&=&\int^{t}_{0}\mbox{sign}(X^1_{s-}-X^2_{s-})(\sigma (X^1_s)-\sigma(X_s^2))dW_s \\& & \\
&+& \int^{t}_{0}\mbox{sign}(X^1_{s-}-X^2_{s-})(b(X^1_s)-b(X_s^2))ds\\ & & \\
&+& \int^{t}_{0}\mbox{sign}(X^1_{s-}-X^2_{s-})\int (F(X^1_{s-},z)-F(X^2_{s-},z))(\mu - \nu )(ds,dz) \\& & \\
\end{array}
$
The first  and the third terms on the right hand sides are martingales and all terms on the right hand side are integrable.\\
We can prove, if we note $I_2(t)$ the second term on the right hand side and  using the fact that $b$ is lipschitz, that
\begin{equation}
	I_2(t)\leq c\int_0^t\left|X^1_s-X^2_s\right|ds \label{8}
\end{equation}
Finally we obtain

	\[ \Bbb{E}\left(\left|X^1_t-X^2_t\right|\right) \leq  c \left(\int^{t}_{0}\Bbb{E}\left|X^1_s-X^2_s\right|ds\right)
\]

and Gronwall lemma implies that $X^1\equiv X^2$.
\end{proof}
\medskip

\begin{proof}[Proof of the corollary 6]

The main idea is to use  the Tanaka formula but with the function \\$x\rightarrow x^+$. Let $X^1$ and $X^2$ two solutions to \eqref{1},  then

\mbox{}

\begin{eqnarray}
	\left(X^1_t-X^2_t\right)^{+}&=&\int^{t}_{0}\textbf{1}_{(X^1_{s-}-X^2_{s-}>0)}d(X^1_{s}-X^2_{s})  \nonumber \\
  &+&  \sum_{0<s\leq t}\textbf{1}_{(X^1_{s-}-X^2_{s-}>0)}\left(X^1_{s}-X^2_{s}\right)^{-} \nonumber \\
&+& \sum_{0<s\leq t}\textbf{1}_{(X^1_{s-}-X^2_{s-}\leq 0)}\left(X^1_{s}-X^2_{s}\right)^{+}\nonumber \\
& + & \frac{1}{2}L^0_t(X^1-X^2) \nonumber         	
\end{eqnarray}

\mbox{}

In order that the ($\mathcal{LT}$) condition be realised, we can prove that the second and the third term on the right hand side is zero. We prove uniquely the second term, the other is proved by  the same method. We can suppose that $x\rightarrow F(x,z)+ x$ is nondecreasing in a neighborhood of $0, \lambda(dz) a.e.$ .
We have

\mbox{}

\begin{equation}
\begin{array}{l}
	\sum_{0<s\leq t}\textbf{1}_{(X^1_{s-}-X^2_{s-}>0)} \left(X^1_{s}-X^2_{s}\right)^{-}\\
	\\
=	\sum_{0<s\leq t}\textbf{1}_{(X^1_{s-}-X^2_{s-}>0)} \left(\Delta X^1_{s}-\Delta X^2_{s}+ X^1_{s-}-X^2_{s-}\right)^{-} \\
	\\
 =\int _0^t\int \textbf{1}_{(X^1_{s-}-X^2_{s-}>0)}\left[(F(X^1_{s-},z)+ X^1_{s-})-(F(X^2_{s-},z)+X^2_{s-})\right]^{-}\mu(ds,dz)
\label{9}        	
\end{array}
\end{equation}

\mbox{}

Since $x+F(x,z)$ is nondecreasing, the right hand size of \eqref{9} is  $0$.
We obtain the result of the first  corollary.
\end{proof}
\medskip

\begin{proof}[Proof of the corollary 7]

The proof of the second corollary is similar to what we have seen previously, this  is due to the hypothesis for $\sigma$ which implies the (LT) condition and the hypothesis for $F$ is identical to that of the first corollary. We can see Engelbert-Schmidt \cite {Eng1,Eng2,Eng3} .  	
\end{proof}
\medskip

\mbox{}

\begin{proof}[Proof of the proposition 8]

We use the Tanaka formula. Let $X^1$ and $X^2$ two solutions to \eqref{1},  then
\begin{equation}
\begin{array}{lll}
	\left|X^1_t-X^2_t\right|&=&\int^{t}_{0}\mbox{sign}(X^1_{s-}-X^2_{s-})d(X^1_{s}-X^2_{s})\\& &\\ & &+ \sum_{0<s\leq t}\left\{\left|X^1_{s}-X^2_{s}\right|-\left|X^1_{s-}-X^2_{s-}\right|\right.\\ & &\\& &\left.-\mbox{sign}(X^1_{s-}-X^2_{s-})\Delta (X^1-X^2)\right\}
	\label{10}       	
\end{array}
\end{equation}
Now, we have
\begin{equation}
\begin{array}{lll}
	d(X^1_{t}-X^2_{t})&=&(\sigma (X^1_t)-\sigma(X_t^2))dW_t + (b(X^1_t)-b(X_t^2))dt \\
	& &\\
	& &+ \int (F(X^1_{t-},z)-F(X^2_{t-},z))(\mu - \nu )(dt,dz)   \label{11}    	
\end{array}
\end{equation}

If we substitute \eqref{11} in \eqref{10}, we obtain

\mbox{}

$\begin{array}{lll}
\left|X^1_t-X^2_t\right|&=&\int^{t}_{0}\mbox{sign}(X^1_{s-}-X^2_{s-})(\sigma (X^1_s)-\sigma(X_s^2))dW_s \\& & \\
&+& \int^{t}_{0}\mbox{sign}(X^1_{s-}-X^2_{s-})(b(X^1_s)-b(X_s^2))ds\\ & & \\
&+& \int^{t}_{0}\mbox{sign}(X^1_{s-}-X^2_{s-})\int (F(X^1_{s-},z)-F(X^2_{s-},z))(\mu - \nu )(ds,dz) \\& & \\
&+& \sum_{0<s\leq t}\left\{\left|X^1_{s}-X^2_{s}\right|-\left|X^1_{s-}-X^2_{s-}\right|-\mbox{sign}(X^1_{s-}-X^2_{s-})\Delta (X^1-X^2)\right\}\\ & &
\end{array}
$
The first  and the third terms on the right hand side are martingales and all terms on the right hand side are integrable.\\
We can prove, if we note $I_2(t)$ the second term on the right hand side and  using the fact that $b$ is lipschitz, that
\begin{equation}
	I_2(t)\leq c\int_0^t\left|X^1_s-X^2_s\right|ds \label{12}
\end{equation}

We treat the fourth term, denoted $I_4(t)$, by the following calculation:

\mbox{}

$\begin{array}{lll}
\left|I_4(t)\right|& \leq & \sum_{0<s\leq t}\left|\left|X^1_{s}-X^2_{s}\right|-\left|X^1_{s-}-X^2_{s-}\right|\right|+\left|\Delta (X^1-X^2)\right| \\ & & \\
& \leq & 2\int^{t}_{0}\int \left|F(X^1_{s-},z)-F(X^2_{s-},z)\right|\mu(ds,dz)\\ & & \\
&=& 2\int^{t}_{0}\int \left|F(X^1_{s-},z)-F(X^2_{s-},z)\right|(\mu-\nu)(ds,dz)\\ & & \\
&+& 2\int^{t}_{0}\int \left|F(X^1_{s-},z)-F(X^2_{s-},z)\right|\nu(ds,dz)\\ & & \\
\end{array}
$

\mbox{}

Taking expectation in the two sides and using the  martingales property and \eqref{5}, we obtain

\mbox{}

$\begin{array}{lll}
\Bbb{E}\left(\left|I_4(t)\right|\right)& \leq & c \Bbb{E}\left(\int^{t}_{0}ds\int \left|F(X^1_{s-},z)-F(X^2_{s-},z)\right|\lambda(dz)\right)\\ & & \\

\end{array}
$

\mbox{}
and hence , we have
\begin{equation}
	\Bbb{E}\left(\left|I_4(t)\right|\right)\leq c \int^{t}_{0}\Bbb{E}\left( \left|X^1_{s}-X^2_{s}\right|\right)ds \label{13}
\end{equation}

By using \eqref{6}, \eqref{13} and Gronwall lemma, we obtain the result of theorem.
\end{proof}
\medskip

\begin{remark}
The proof of the proposition 6 is simpler than the one proposed by Höpfner \cite {Hopfner} especially for the last term.
\end{remark}

\begin{remark}
We can prove the theorem with a weaker hypothesis for $b$ and $F$
\[\left|b(x)- b(y)\right|^2+\int \left|F(x,z)- F(y,z)\right|^2 \lambda(dz)\leq  h^2(\left|x-y\right|)
\]
where $h$ verifies the same assumptions as in corollary 5.

\end{remark}


\section{Others results}

\begin{theorem}

 If $(\sigma,F)$ verify ($\mathcal{LT}$) condition and $b$ is measurable and bounded. Then the uniqueness in the sense of probability law  implies the pathwise uniqueness.
\end{theorem}
\begin{proof}
[Proof of theorem 11]
Let $X^1$ and $X^2$ two solutions of equation \eqref{1}. We will prove that $Y=X^1 \vee X^2$  and $Z=X^1 \wedge X^2$ are  solution of \eqref{1}. We obtain by using Tanaka formula and  ($\mathcal{L}$T) condition ,

\mbox{}

$\begin{array}{lll}
\left(X^2_t-X^1_t\right)^+&=& \int^{t}_{0^+}\textbf{1}_{(X^2_{s-}-X^1_{s-}>0)}d(X^2_s-X^1_s) \\
\end{array}
$

\mbox{}

by using  the fact that $X^1 \vee X^2=X^1 +(X^2-X^1)^+$ we obtain

\mbox{}

$\begin{array}{lll}
X^1_t \vee X^2_t &=&X^1_0 + \int^{t}_{0}\left[\textbf{1}_{(X^2_{s-}-X^1_{s-}>0)}(\sigma (X^2_s)-\sigma(X^1_s))+\sigma(X^1_s)\right]dW_s \\& &\\&+& \int^{t}_{0}  \left[\textbf{1}_{(X^2_{s-}-X^1_{s-}>0)}(b(X^2_s)-b(X^1_s))+ b(X^1_s)\right]ds \\& &\\&+& \int^{t}_{0^+}\int  \left[\textbf{1}_{(X^2_{s-}-X^1_{s-}>0)}(F(X^2_s,z)-F(X^1_s,z))+ F(X^1_s,z)\right](\mu-\nu)(ds,dz)\\& &\\&=& X^1_0 + \int^{t}_{0}\sigma (X^1_s \vee X^2_s)dW_s + \int^{t}_{0}  b(X^1_s \vee X^2_s)ds  \\& &\\&+& \int^{t}_{0^+}\int F((X^1 \vee X^2)_{s-},z)(\mu - \nu)(ds,dz)\\&\\
\end{array}
$

\mbox{}

Then $Y$ is a solution of \eqref{1}.\\
We have on the other hand,\\
$
\begin{array}{lll}
\left(X^1_t-X^2_t\right)^+&=& \int^{t}_{0^+}\textbf{1}_{(X^1_{s-}-X^2_{s-}>0)}d(X^1_s-X^2_s) \\
\end{array}
$

\mbox{}

by using  the fact that $X^1 \wedge X^2=X^1 -(X^1-X^2)^+$ we obtain by the same way that

\mbox{}

$
\begin{array}{lll}
X^1_t \wedge X^2_t&=& X^1_0 + \int^{t}_{0}\sigma (X^1_s \wedge X^2_s)dW_s + \int^{t}_{0}  b(X^1_s \wedge X^2_s)ds  \\& &\\&+& \int^{t}_{0^+}\int F((X^1 \wedge X^2)_{s-},z)(\mu - \nu)(ds,dz)\\
\end{array}
$

\mbox{}

Then $Z$ is a solution of \eqref{1}.\\
Finally, we have for all $t\geq 0$
	\[\Bbb{E}\left[\left|X^1_t-X^2_t\right|\right]=\Bbb{E}\left[X^1_t\vee X^2_t\right]-\Bbb{E}\left[X^1_t\wedge X^2_t\right]
\]
and  by using the  uniqueness in the sense of probability law, we obtain
	\[\Bbb{E}\left[\left|X^1_t-X^2_t\right|\right]=0
\]
Since $X^1$ and $X^2$ are càdlàg, hence $X^1_t=X^2_t$ for all $t\geq 0$ a.s.
\end{proof}

\mbox{}
This allows us to give a generalisation of Bass's result \cite{Bass 1}, we have the following:

\mbox{}
\begin{theorem}
If $(\sigma,F)$ verify ($\mathcal{LT}$) condition and $b$ is measurable and bounded. Moreover we suppose that:
\begin{itemize}
	\item $\sigma$ is bounded and continuous and strictly positive
	\item $x \rightarrow \int_A \frac{\left|z\right|^2}{1+\left|z\right|^2}F(x,z)\lambda(dz)$ is bounded and continuous for each $A \subset \Bbb{R}-\left\{0\right\}$
	\end{itemize}
Then there exists a solution to \eqref{1} and that solution is pathwise unique.
\end{theorem}
\mbox{}

If we set $F(x,z)=\frac{1}{\left|z\right|^{1+\alpha(x)}}$ and $\sigma = b= 0$ in the equation \eqref{1}, we obtain the stable-like process with the operator

	\[
	\mathcal{L}f(x)=\int  \left[f(x+z)-f(x)-\textbf{1}_{(\left|z\right|\leq 1)}f'(x)z\right] \frac{1}{\left|z \right|^{1+\alpha(x)}}dz
 \]
We have the following proposition
 \begin{proposition}
 If the function $\alpha$ is Dini continuous, bounded above by a constant less than 2 and bounded bellow by a constant greater than 0  and is increasing, then pathwise uniqueness of solution of stochastic differential equation driven by stable-like process associated to $F$.
\end {proposition}
\mbox{}
The proof is a consequence of Bass result (\cite{Bass 1},p.13).
\mbox{}
\mbox{}

\begin{theorem}

 Suppose that for $i=1,2$, $X^i$ satisfies:

 \begin{equation}
		dX^i_t=\sigma(X^i_t)dW_t+ b_i(X^i_t)dt + \int F(X^i_{t-},z)(\mu-\nu)(dt,dz) \label{12}
\end{equation}

 where $\sigma,F,b_1,b_2,$ are bounded measurable functions. Assume that:

\begin{itemize}
	\item $(\sigma,F)$ verify ($\mathcal{LT}$) condition
	\item One of the two functions $b_1,b_2$ is Lipschitz
\end{itemize}
Assume further that:
\begin{enumerate}
	\item $b_1 \leq b_2$
	\item $X^1_0 \leq X^2_0$
	\end{enumerate}
Then: $X^1_t \leq X^2_t$ for all $t$ a.s.

\end{theorem}

\begin{proof}
[Proof of the Theorem 14]

Let $X_i, i=1,2$ two solutions of equations \eqref{12}. By Tanaka formula we obtain

\mbox{}

\begin{equation}
\begin{array}{lll}
	\left(X^1_t-X^2_t\right)^{+}& = & \int^{t}_{0}\textbf{1}_{(X^1_{s-}-X^2_{s-}>0)}d(X^1_{s}-X^2_{s}) \\
	& &\\
	 & + & \sum_{0<s\leq t}\textbf{1}_{(X^1_{s-}-X^2_{s-}>0)} \left(X^1_{s}-X^2_{s}\right)^{-}\\
	 & & \\
	&+& \sum_{0<s\leq t}\textbf{1}_{(X^1_{s-}-X^2_{s-}\leq 0)}\left(X^1_{s}-X^2_{s}\right)^{+}\\
	& & \\
	 & + & \frac{1}{2}L^0_t(X^1-X^2)\\	
\end{array}
\label{13}
\end{equation}

As $(\sigma,F)$ verify ($\mathcal{LT}$) condition, the third, the fourth and the fifth terms in the right hand sides are zero. Using the same argument as before, we find
	\[\Bbb{E}\left[\left(X^1_t-X^2_t\right)^{+}\right]\leq c \int^{t}_{0}\Bbb{E}\left[\left(X^1_t-X^2_t\right)^{+}\right]ds
\]
which implies that $X^1_t\leq X^2_t$ for all $t\geq 0$ a.s.(since $X^1$ and $X^2$ are càdlàg).
\end{proof}

\begin{theorem}

 Suppose that for $i=1,2$, $X^i$ satisfies:

 \begin{equation}
		dX^i_t=\sigma(X^i_t)dW_t+ b_i(X^i_t)dt + \int F_i(X^i_{t-},z)(\mu-\nu)(dt,dz) \label{14}
\end{equation}

 where $\sigma,F_1,F_2,b_1,b_2,$ are bounded measurable functions. Assume that:

\begin{itemize}
	\item $\sigma$ verify $(LT)$ condition
	\item One of the two functions $b_1,b_2$ is Lipschitz
\end{itemize}
Assume further that:
\begin{enumerate}
	\item $b_1 \leq b_2$
	\item $X^1_0 \leq X^2_0$
	\item $x_1+F_1(x_1,z)\leq x_2+F_2(x_2,z)$ for all $x_1\leq x_2$ .
	\end{enumerate}
Then: $X^1_t \leq X^2_t$ for all $t$ a.s.

\end{theorem}

\begin{remark}

We can see that if $F_i,i=1,2$ satisfy the hypothesis of Corollary 4 ( or 5) , then we can replace the hypothesis 3. by
	\[F_1(x_1,z) \leq  F_2(x_2,z),\forall x_1 \leq x_2 .
\]
\end{remark}

\begin{remark}
If we take $F_1=F_2$ in the theorem 12, the assumption of the corollary 4 ( or 5) on $F$ is enough for the conclusion  of the theorem 14.
\end{remark}

\begin{remark}
We can remark that the condition 3 in the theorem 15 ensures that the paths of  $X_i$ do not cross at jump times: if $(s,z)$ is  an atom of $\mu$  and if $X^1_{s-}=x\leq X^2_{s-}=y $, then
	\[X^1_{s}=x + F_1(x,z)\leq y + F_2(y,z)=X^2_{s}
	\]
and this condition is necessary for comparaison theorem.
\end{remark}

\begin{proof}
[Proof of Theorem 15]

We apply the same method as before in the theorem 14, the third and the fourth terms in the right hand side of \eqref{13} is zero because of the assumption 3 and the fifth term is zero because the $\sigma$ verify $(LT)$ condition. The end of the proof is identical to that of the theorem 14.
\end{proof}
\medskip

\begin{remark}
We can remark that the theorem 15 generalizes the result of Peng and Zhu \cite{Peng} namely theorem 3.1.
\end{remark}


\section{Case of SDE driven by spectrally positive Lévy process}
In this section we consider a SDE driven by spectrally positive Lévy process. Let $Z_t$ a spectrally positive stable process whith exponent $\alpha \in (1,2)$ that is $Z$ can be written as
	\[Z_t=\int^{t}_{0}\int^{\infty}_{0}z (\mu(ds,dz)-\nu(ds,dz))
\]
where $\nu$ is a measure defined by
	\[\nu(ds,dz)=ds\left|z\right|^{-\alpha - 1}dz
\]
and $\mu(ds,dz)$ is a Poisson measure on $\left[0,\infty \right)\times \left.\right]0,\left.\infty \right)$.\\
We consider the stochastic differential equation
\begin{equation}
		dX_t=G(X_{t-})dZ_t \label{16}
\end{equation}
In paper of \cite{LiMytnik}, the authors prove pathwise uniqueness of equation \eqref{16} under restrictive assumptions.
We can prove the same result with less restrictives assumptions.

\begin{proposition}

Let $G$ be a non-decreasing function on $\Bbb {R}_+$ satisfying \\$G(0)=0$.\\
Then the pathwise uniqueness holds for solutions of \eqref{16}.
\end{proposition}

\begin{proof}
We apply corollary 6 by noticing that the function $F$ in this case is written $F(x,z)=zG(x)$.The nondecreasing of $x\rightarrow G(x)$ implies the nondecreasing of $x \rightarrow x+F(x,z)$ a.e.$\lambda(dz)$.
\end{proof}
\begin{remark}
Our approach based on local time technic can be used for more general equation of type:
\begin{eqnarray*}
X_{t}=X_{0}+ \int_{0}^{t} \sigma(X_{s^{-}},u)\;W(ds,du)+ \int_{0}^{t}b(X_{s^{-}})\;ds\\
+ \int_{0}^t \int_{|u| \leq 1}g_{0}(X_{s^{-}},u)\widetilde{N}_{0}(ds,du)+ \int_{0}^t \int_{u \geq 1}g_{1}(X_{s^{-}},u)N_{1}(ds,du)
\end{eqnarray*}
where
\begin{itemize}
  \item $\{W(ds,du)\}$ the white noise with intensity $ds \pi(dz)$, with $\pi$ is a $\sigma-$finite measure on $\mathbb{R}$.
  \item $N(ds,du)$ and $\widetilde{N}(ds,du)$ denote the Poisson random measures on $[0,\infty)\times [-1,1]$, $[0,\infty)\times [-1,1]^{c}$ respectively, defined on same probability space and are independent each of other.
  \item $\widetilde{N}_{0}(ds,du)$ denote the compensated measure of $N_{0}(ds,du)$.
\end{itemize}
This SDE was recently studied by Donald A. Dawson and Zenghu Li \cite{Dawson} and treated by Y. Ouknine \cite{Ouknine} in continuous case, it will be the subject of another paper.
\end{remark}
\end{spacing}


\begin{thebibliography}{9}
\bibitem{bandroff} O. E. Barndorff-Nielsen (1998): Processes of normal inverse Gaussian type.
Finance and Stochastics. 2 41-68.
\bibitem{Barlow} M. T. Barlow et E. Perkins. SDE's singular increasing process, Stochastics, $12$ $(1984)$, $229-242$.
\bibitem {Bass 1} Bass, R.: Stochastic differential equations with jumps. \\Probability Surveys \textbf{1}, (2004),1-19.

\bibitem {Bass 2}Bass, R.: Stochastic differential equations driven by symmetric stable processes, Séminaire de probabilités, XXXVI, Lect. Notes in Math. \textbf{1801}, (2003),302-313.

\bibitem {Bel} Belfadli, R.,Ouknine, Y.:  On the pathwise uniqueness of solutions of stochastic differential equations driven by symmetric stable Lévy processes, Stochastics An International Journal of Probability and Stochastic Processes,  \textbf{80:6}, (2008), 519-524.
\bibitem{Benabd} M. Benabdallah, S. Bouhadou and Y. Ouknine. On Semimartingale local time inequalities and Applications in SDE's. Preprint.
\bibitem{Dawson}  Donald A. Dawson and Zenghu Li: Stochastic equations, flows and measure-valued processes. Preprint (2011)
\bibitem {Eng1} H.J. Engelbert, W. Schmidt: Strong Markov Continous Local Martingales and Solutions of One-Dimensional Stochastic Differential Equations, I. Math. Nachr. \textbf{143}, 167–184 (1989).
\bibitem{Rama} R. Cont, and P.Tankov: (2004), Financial Modeling with Jump Processes, Chapman and
Hall/CRC, Financial Mathematics Series.
\bibitem {Eng2} H.J. Engelbert, W. Schmidt: Strong Markov Continous Local Martingales and Solutions of One-Dimensional Stochastic Differential Equations, II. Math. Nachr. \textbf{144}, 241,281 (1989).

\bibitem {Eng3} H.J. Engelbert, W. Schmidt: Strong Markov Continous Local Martingales and Solutions of One-Dimensional Stochastic Differential Equations, III. Math. Nachr. \textbf{151}, 149–197 (1991).


\bibitem {Fu} Fu, Z.,Li, Z.: Stochastic equations of non-negative processes with jumps, Stochastic Processes and their Applications, \textbf{120}, (2010), 306-370.


\bibitem {Hopfner} Höpfner, R.: An extension of the Yamada-Watanabe condition for pathwise uniqueness to stochastic differential equations with jumps, Elect. Comm. in Probab., \textbf{14}, (2009), 447-456.
\bibitem{Ikeda}N. Ikeda and S. Watanabe: Stochastic differential equations and diffusion processes.
2nd ed. North-Holland / Kodansha 1989.
\bibitem {Le Gall}  Le Gall J. F.: One-dimensional stochastic differential equations involving the local-times of unknown processes, Stochastic Analysis and Applications, Lect. Notes in Math. \textbf{1095}, (1983), 51-82.

\bibitem {LiMytnik}  Li Z.,Mytnik L.:Strong solutions for stochastic differential equations with jumps, Annales de l'Institut Henri Poincare: Probabilites et Statistiques.,To appear, (2010).
\bibitem{Nakao} Y. Ouknine, Généralisation d’un lemme de S. Nakao et applications. Stochastics 23
(1988), no. 2, 149–157.
\bibitem{fsemimarting} Y. Ouknine, Fonctions de semimartingales et applications aux équations différentielles
stochastiques. Stochastics 28 (1989), no. 2, 115–122.
\bibitem{oukref} Y. Ouknine, Quelques identités sur les temps locaux et unicité des solutions
d’équations differentielles stochastiques avec reflection. Stochastic Process. Appl. 48
(1993), no. 2, 335–340.
\bibitem {Ouknine}  Ouknine Y. : Unicité trajectorielle	des equations différentielles stochastiques avec temps local, Probab. and Math. Stat., \textbf{19}, Fasc. 1 (1999), 55-62.

\bibitem {Peng} Peng, S.,Zhu, X.: Necessary and sufficient condition for comparison theorem of 1-dimensional stochastic differential equations, Stochastic Processes and their Applications, \textbf{116}, (2006), 370-380.

\bibitem {Protter} Protter, P.: Stochastic integration and differential equations. 2nd ed. Springer 2005.
\bibitem{Ornstein}L. S. Ornstein and G. E. Uhlenbeck: (1930). Physical Review. 36 823-841.
\bibitem{Rutkowski} M. Rutkowski, Stochastic differential with singular drift. Statistics and Probability
letters 10 (1990) 225-229.
\bibitem{samwe}P. A. Samuelson:(1965). Rational theory of warrant pricing. Industrial Management
Review. 6 13-31.
\bibitem{skoro} A. Skorokhod: Studies in the theory of random processes. Addison-Wesley 1965.
\end{thebibliography}
\end{document}